\documentclass[a4paper,11pt,openany,reqno]{amsart}
	\usepackage[latin1]{inputenc}
    \usepackage[T1]{fontenc}
    \usepackage[english]{babel}
    \usepackage{appendix}
    \usepackage[english]{minitoc}
    \usepackage[nice]{nicefrac}
\usepackage[top=2.5cm, bottom=2.5cm, left=2.5cm, right=2.5cm]{geometry}
\numberwithin{equation}{section}
\makeatletter\@addtoreset{equation}{section} 
	\usepackage{amsmath}
    \usepackage{amssymb}
    \usepackage{amsbsy} 
    \usepackage{latexsym, amsfonts}
    \usepackage{graphics}
    \usepackage{ulem}
    \usepackage{hhline}
    \usepackage{dsfont}
    \usepackage{mathrsfs}
    \usepackage{color}
    \usepackage{fancyhdr}
    \usepackage{rotating}
    \usepackage{fancybox}
    \usepackage{colortbl}
    \usepackage{pifont}
    \usepackage{setspace}
    \usepackage{enumerate}
    \usepackage{multicol}
    \usepackage{varioref}
    \usepackage{lmodern}
    \usepackage{textcomp}
    \usepackage{euscript}
    \usepackage[pdftex]{hyperref}
			\newtheorem{theorem}{Theorem}[section]

			\newtheorem{proposition}[theorem]{Proposition}

			\newtheorem{remark}[theorem]{Remark}

\newcommand{\C}{\mathbb C}
 
  \newcommand{\Z}{\mathbb Z} 	 
  \newcommand{\Hq}{\mathbb H}
   
    \newcommand{\Cc}{\mathcal{C}_\mu}
    
   \newcommand{\scal}[1]{{\left\langle{#1}\right\rangle}}
    \newcommand{\bz}{\overline{z}}  \newcommand{\bw}{\overline{w}}

\newcommand{\LgC}{L^{2}_\mu}
\begin{document}
\title{On the range of weighted planar Cauchy transform}
\thanks{Dedicated to the memory of Ahmed Intissar passed away in July 26, 2017}

\author{A. Ghanmi}
\email{allal.ghanmi@um5.ac.ma/ag@fsr.ac.ma}

\address{Analysis, P.D.E $\&$ Spectral Geometry, Lab M.I.A.-S.I., CeReMAR, 
	Department of Mathematics,\newline
	P.O. Box 1014,  Faculty of Sciences, 
	Mohammed V University in Rabat, Morocco}
\maketitle

\large

\begin{abstract}
 We describe the range of of weighted Cauchy transform and its $k$-Bergman projection when action on weighted true poly-Bargmann spaces constituting an orthogonal Hilbertian decomposition of the Hilbert space of Gaussian functions on the complex plane.
\end{abstract}

\section{Introduction}
 The Cauchy transforms on bounded regular domains and their boundaries are well studied and has been investigated by many authors, see for instance \cite{AnderssonHinakkenen1989,AndersonKhavinsonLomonosov1990,ArazyKhavinson1992,Dostanic1996,Tolsa1999,Kalaj2012,HruscevVinogradov1981,Brennan2004,CimaMathesonRoss2006,Bell2016} and the references therein. This rich literature is due to their use in solving the $\overline{\partial}$-equation,
  in developing theory of holomorphic functions  \cite{Hormander1990} and in proving interpolation theorems \cite{Jones1983} as well as in providing simple proof of Corona theorem \cite{Gamelin1980}.
One of the fundamental examples of weighted Cauchy transform on the whole complex plane is that associated to the Gaussian measure $d\mu= e^{-|z|^2}dxdy$ and defined by
 \begin{align}\label{CT}
 \Cc f(z) :=  \frac 1\pi \int_{\C} \frac{f(\xi)}{z-\xi} d\mu(z)
 \end{align}
 on $\LgC$, denoting as usual the Hilbert space $L^2(\C,d\mu)$ of all square integrable complex valued functions with respect to the scalar product
 $$\scal{f,g}:=  \int_{\C} f(z)\overline{g(z)}   d\mu(z) .$$
 The singular integral operator in \eqref{CT}, as operator from $\LgC$ into $\LgC$, is  bounded, compact and belongs to the $p$-Schatten class for every $p>2$.
The spectral properties of $\Cc$ has been investigated in  \cite{Dostanic2000,In}.
In \cite{Dostanic2000}, Dostani\'c gave
the exact asymptotic behavior of the singular values of $\Cc$ and $P \Cc$, where $P$ denotes the orthogonal projection operator onto the classical Bargmann space $A^2$ constituted of all holomorphic functions belonging to $\LgC$.
The generalization to the polyanalytic setting was considered by the brothers A. and A. Intissar  in \cite{In}.
This was possible making use of the  Hilbertian orthogonal decomposition $\LgC=\bigoplus \limits_{n=0}^{+\infty}\mathcal{A}^2_n,$
where $\mathcal{A}^2_n=Ker(\Delta - n)$ are the $L^2$--eigenspaces of the Landau operator $\Delta = \partial_z\partial_{\bz} - \bz \partial_{\bz}$ with $A^2_0=A^2$.
The key result in \cite{In} is the explicit action of $\Cc$ in \eqref{CT} on the It\^o--Hermite polynomials constituting an orthogonal basis of $\LgC$.
	In fact, the functions $\Cc H_{m,n}$ are given by
	\begin{align}\label{actionCH} (\Cc H_{m,n})(z)
	= - e^{-|z|^2} H_{m-1,n}(z;\bz).
	\end{align}
	Moreover,  for varying $m=0,1, \cdots,$ and fixed $n$, (resp.for varying $m=0,1, \cdots,$ and fixed $m$) they constitute an orthogonal system in $\LgC$.

In the present note, we complete the study provided in the afore mentioned papers \cite{Dostanic2000,In} related to $\Cc$ in  $\LgC$. Mainly, we provide complete description of which polyanalytic functions on the whole complex plane can be represented as weighted Cauchy integral in \eqref{CT}. This leads to the  identification of the range of the operator
$P_n\Cc $, where $P_n$ denotes the orthogonal projection on $\mathcal{A}^2_n$ given by
\begin{align} \label{OrthProjeq}
P_nf(z) &= \dfrac{1}{\pi} \int_{\C}   L_{n}(|z - \xi|^2 )   e^{\bz \xi} f(\xi) d\mu(\xi) .
\end{align}
To explore these ideas, we begin by reviewing in Section 2 the basic properties of the
It\^o--Hermite polynomials and their associated weighted poly-Bargmann spaces $\mathcal{A}^2_n$. Our main results on the range of of weighted Cauchy transform and its $k$-Bergman projection on  $\mathcal{A}^2_n$ and $\LgC$ are presented and proved in Section 3.

 \section{Preliminaries}
 The It\^o--Hermite polynomials involved in  \eqref{actionCH}
 are the basic example of bivariate real Hermite polynomials that are not simply a tensor product of univariate
 real Hermite polynomials. They are due to It\^o \cite{Ito52} and play a crucial role in the framework of complex Markov process and constitute an orthogonal basis of $\LgC$. They have been intensively studied and found several applications in the nonlinear analysis of traveling wave tube amplifiers \cite{Barrett}, in spectral theory of some second order differential operators \cite{Shigekawa87,Matsumoto96,Gh2018Mehler}, in the study of some special integral transforms \cite{In,BenElhGh2019}, in coherent states theory \cite{AliBagarelloHonnouvo10,AliBagarelloGazeau13}, combinatorics \cite{Ismail2015Proc} and in signal processing \cite{RaichZhou04,DallingerRuotsalainenWichmanRupp10}.
  They are defined by their Rodrigues' formula   \cite{Ito52}
 \begin{align}\label{chp}
 H_{m,n}(z,\bz)&:
 = (-1)^{m+n}   e^{ |z|^2 } \partial_z^{m} \partial_{\bz}^{n} \left(e^{-  |z|^2 }\right),
 \end{align}
 where $\partial_z$ and $\partial_{\bz}$, as usual,  denote the first order partial differential operators
 \begin{align}\label{dbar}
 \partial_z = \frac{\partial}{\partial z} :=
 \dfrac{1}{2}\left(\frac{\partial }{\partial x} - i\frac{\partial }{\partial y}\right), \quad \partial_{\bz} = \frac{\partial}{\partial \bz} :=
 \dfrac{1}{2}\left(\frac{\partial }{\partial x}+i\frac{\partial }{\partial y}\right).
 \end{align}
 The hypergeometric representation for $H_{m,n}$ in terms of the Kummer's function  ${_1F_1}$ reads
 \begin{align}
 H_{m,n}(z,\bz) &=c_{m,n} \frac{z^m\bz^n}{|z|^{2\min(m,n)}}  {_1F_1}\left( \begin{array}{c} -\min(m,n) \\ |m-n|+1 \end{array}\bigg | |z|^{2}  \right) \label{16negative}
\\
 &=\left\{\begin{array}{lll}
 \dfrac{(-1)^{m} n! }{(n-m)!}
 \overline{z}^{n-m}  {_1F_1}\left( \begin{array}{c} -m\\ n-m+1 \end{array}\bigg | |z|^{2} \right); \quad m\leq n\\
 \\
 \dfrac{(-1)^{n} m!}{(m-n)!}
 z^{m-n}  {_1F_1}\left( \begin{array}{c} -
 n \\ m-n+1 \end{array}\bigg |  |z|^{2} \right) ; \quad m\geq n ,
 \end{array} \right. \label{HypergeometricRep}
 \end{align}
 where $m\wedge n= \min(m,n)$, $m\vee n = \max(m,n)$ and
  \begin{align}\label{constcmn}c_{m,n}:= \dfrac{(-1)^{\min(m,n)}\max(m,n)!}{|m-n|!}.
  \end{align}
 This last representation is the convenient one  for extending $H_{m,n}$  to include negative index, leading to what we call here extended It\^o--Hermite functions. 
For further basic properties of $H_{m,n}$, we refer to \cite{In,Gh13ITSF,Gh08JMAA,Gh2018Mehler,Ismail2015Proc,IsmailTrans2016}.

 By considering the formal adjoint of $\partial_{\bz}$ in $\LgC$, given by $\partial_{\bz}^*=-\partial_z + \bz$, we can see that the operator $\Delta = \partial_{\bz}^*\partial_{\bz}= -\partial_{z} \left( \partial_z - \bz\right)  \partial_{\bz} $ is the usual Landau magnetic Laplacian  describing the movement of a single charged particle in the complex plane under the action of an uniform magnetic field applied perpendicularly. The concrete spectral analysis of such Laplacian is well known in the literature, see for instance \cite{AskIntmou2000} and the references therein. Thus, the $L^2$-eigenspaces
 $ \mathcal{A}_{n}^{2} : = \ker(\partial_{\bz}\partial_{\bz}^* -nId)$ form an orthogonal Hilbertian decomposition of $\LgC$. More explicitly, they are characterized as the closed Hilbert subspaces of all convergent series
 \begin{equation}\label{ExpGC} f(z) = \sum_{j=0}^{+\infty} H_{j,n}(z,\bz)\alpha_{j,n}, \quad  \sum_{j=0}^{+\infty} j!|\alpha_{j,n}|^2 <+\infty.
  \end{equation}
Moreover, they are exactly
the so-called true poly-Bargmann spaces in the terminology of Vasilevski \cite{Vasilevski2000} realized  as specific closed subspace of $(n+1)$-polyanalytic $L^2$-functions.
We conclude this section by noticing that $ \mathcal{A}_{n}^{2}$
is a reproducing kernel Hilbert space and the $\{H_{m,n}; m=0,1,2,\cdots\}$ is an orthogonal system of it. Thus, the expansion series of its reproducing kernel function $w \longmapsto \mathcal{K}_{n}(z,w) =\overline{\mathcal{K}_{n}(w,z)}
\in \mathcal{A}_{n}^{2}$ in terms of the  It\^o--Hermite polynomials reads
\begin{align} \label{expKernel} \mathcal{K}_{n}(z,w) =  \sum_{m=0}^\infty \frac{H_{m,n} (z,\bz) H_{n,m} (w,\bw)}{\pi m!n!} ,
\end{align}
since $ H_{m,n} (z,\bz)/\sqrt{\pi m!n!}$ form an orthonormal basis of $\mathcal{A}_{n}^{2}$,
while the closed expression of $\mathcal{K}_{n}$ in terms of the Laguerre polynomials is given by \cite{AskIntmou2000}
\begin{align}
\mathcal{K}_{n}(z,w)&= \dfrac{e^{\bz w}}{\pi}  L _n(|z - w|^2 ) ,
\end{align}
so that the orthogonal projection $f \longmapsto P_nf$ of $L^2(\C,d\mu)$ onto $\mathcal{A}_{n}^{2}$ reads
$P_nf(w) = \scal{\mathcal{K}_{n}(\cdot,w),f} .$

  \section{The range and the null space of $\Cc$}

We begin by noticing that
\eqref{actionCH} shows clearly that for every $m;n=0,1,2,\cdots ,$ the image $\psi_{m,n} :=\Cc H_{m,n}$ belongs to $\LgC$. Therefore, the functions $\psi_{m,n}$ can be expanded, according to the orthogonal decomposition of $\LgC$ in terms of $\mathcal{A}_{n}^{2}$, as
$$\psi_{m,n}  = \sum_{n= 0}^{+\infty} \widetilde{\psi}_{m,n} $$
with the component $\widetilde{\psi}_{m,n}$ are given by $\widetilde{\psi}_{m,n}:= P_n \psi_{m,n} \in \mathcal{A}_{n}^{2} $.
Accordingly, our first aim below is to look  for the explicit expression of $\widetilde{\psi}_{m,n}$. To this end, we use $ \varepsilon_{p} $ to mean $1$ for nonnegative integer $p$ and $0$ otherwise.

\begin{proposition} \label{propProjPnC}
	We have
	\begin{equation} \label{actionPnH}
	P_n(\Cc H_{j,k} )(z)
	= \varepsilon_{n+j-k-1} \frac{(-1)^{n +  k} (n+j-1)! }{ 2^{n+j}n! (n+j-k-1)!}  H_{n+j-k-1,n} (z,\bz)  .
	\end{equation}
\end{proposition}

\begin{proof}
	From the expansion series of the reproducing kernel $\mathcal{K}_{n}$ of $\mathcal{A}_{n}^{2}$ in terms of the  It\^o--Hermite polynomials given through \eqref{expKernel}
	and the explicit expression of $\psi_{j,k}=\Cc H_{j,k}$ given through \eqref{actionCH},
	it follows
	\begin{align*}
	P_n(\psi_{j,k} )(z) 
	&= -\sum_{m=0}^\infty \frac{H_{m,n} (z,\bz)}{\pi m!n!}  \int_{\C} H_{n,m} (w,\bw) H_{j-1,k}(w,\bw) e^{-2|w|^2} d\lambda(w).
	\end{align*}
	The hypergeometric representation of $H_{j,k}$ infers
	\begin{align*}
P_n(\psi_{j,k})(z)
&= \sum_{m=0}^\infty  H_{m,n} (z,\bz) J_{m,n,j,k}
\end{align*}
	where we have set
		\begin{align*}
	J_{m,n,j,k} := -\frac{c_{n,m}c_{j-1,k}}{\pi m!n!}\int_{\C} \frac{w^n\bw^m w^{j-1}\bw^k}{|w|^{2(m\wedge n + (j-1)\wedge k)}} R_{m+1,n,j,k}(|w|^2) 	e^{-2|w|^2} d\lambda(w)
	\end{align*}
	and
	\begin{equation} \label{ptodHyp}
	R_{m,n,j,k}(t) :=  {_1F_1}\left( \begin{array}{c} - (m-1)\wedge n \\ |m-1-n|+1\end{array}\bigg | t \right)
	{_1F_1}\left( \begin{array}{c} -(j-1)\wedge k \\ |j-1-k|+1\end{array}\bigg | t \right) .
	\end{equation} 	
	By passing to polar coordinate, the expression of $J_{m,n,j,k}$ reduces to
	\begin{align*}
	J_{m,n,j,k}
	&= - \frac{2c_{n,m}c_{j-1,k}}{ m!n!}  \delta_{n+j,m +k+1}  \int_0^{\infty}  \frac{r^{n+m +j-1 +k}}{r^{2(m\wedge n + (j-1)\wedge k)}}  R_{m+1,n,j,k}(r^2)
	e^{-2r^2} rdr.
	\end{align*}
	Therefore, the only nonzero term in such expansion of $ P_n(\Cc H_{j,k} )(z)$,  when $n+j\geq k+1$,
	corresponds to the special case  $m=n+j-k-1$. Otherwise, $ P_n(\Cc H_{j,k} )(z)=0$.  Thus, we have
	\begin{align} \label{actionPnHpf}
	P_n(\psi_{j,k} )(z)
	&=\varepsilon_{n+j-k-1}J_{n+j-k-1,n,j,k} H_{n+j-k-1,n} (z,\bz)     .
	\end{align}
The occurring integrals are clearly convergent. For the explicit  computation of $J_{n+j-k-1,n,j,k}$
	can be handled by distinguishing two cases $j\geq k+1$ and $ j\leq k+1$, and  making appeal to the integral formula for the product of two confluent hypergeometric function  \cite[p. 293]{MagnusOberhettingerSoni1966},
	and the
	Gauss's theorem giving the special value of the Gauss  hypergeometric function ${_2F_1}$ at $x=1$.
	Indeed, we have
	\begin{align*}
	J_{n+j-k-1,n,j,k}
	&= -\frac{c_{n+j-k-1,n}c_{j-1,k}}{ (n+j-k-1)!n!} \int_0^{\infty}   t^{|j-k-1|} 	R_{n+j-k,n,j,k}(t)
	e^{-2t} dt\\
	&= -\frac{c_{n+j-k-1,n}c_{j-1,k}}{ (n+j-k-1)!n!}      \frac{\Gamma(|k+1-j|+1)}{2^{  (n+j-k-1)\wedge n   + (j-1)\wedge k +|k+1-j| +1}}
	\\& \qquad \times {_2F_1}\left( \begin{array}{c} - (n+j-k-1)\wedge n ,  - (j-1)\wedge k\\ |k+1-j|+1 \end{array}\bigg | 1\right) \\
	&=\frac{(-1)^{n +  k+1}  \Gamma(j+ n   ) }{ 2^{n+j}n! \Gamma(-k+j+  n )}.
	\end{align*}
	Finally, inserting this in  \eqref{actionPnHpf} yields \eqref{actionPnH}.
\end{proof}

The following result describes the range $R^\ell_n = P_n\Cc( \mathcal{A}_{\ell}^{2})$ of the restriction of  $P_n\Cc$ to the true poly-Bargmann  space  $\mathcal{A}_{\ell}^{2}$ for given  nonnegative integer $\ell$. We also consider  $\widetilde{R}^\ell_n = P_n\Cc( \widetilde{\mathcal{A}}_{\ell}^{2})$, where $\widetilde{\mathcal{A}}_{\ell}^{2}:= \overline{Span\{H_{l,n}; \, n=0,1,2,\cdots \}}$.

\begin{theorem}
	The following assertions hold trues
	\begin{enumerate}
		\item[i)] The space $R^\ell_n$ is an infinite vector space  spanned by
		$  H_{n+j-\ell-1,n} $; $j\geq \max(0,\ell+1-n).$
		\item[ii)] The space $\widetilde{R}^0_0$ is trivial, $\widetilde{R}^0_0 =\{0\}$.
		\item[iii)] For $\ell+n>0$, $\widetilde{R}^\ell_n$ is a finite dimensional vector space of dimension $n+\ell$. An orthogonal basis of $\widetilde{R}^\ell_n$ is given by
		$  H_{k,n} (z,\bz)$; $k=0,1, \cdots, n+\ell-1 .$
	\end{enumerate}
\end{theorem}

\begin{proof}
	Starting from the expansion in \eqref{ExpGC} for for given $f\in \LgC$,
we can write
	\begin{align*}
	P_n\Cc f(z)
	&= \sum_{j,k=0}^\infty P_n (\Cc H_{j,k})(z) \alpha_{j,k}.
	\end{align*}
	Thus,
	by  rewriting the double summation as	
	$$\sum_{j,k=0}^\infty   = \sum_{j<0\vee (\ell+1-n)}   \sum_{k=0}^{\infty} + \sum_{j = 0\vee (\ell+1-n)}^\infty \sum_{k=0}^{ n+j-1}   + \sum_{j=0\vee (\ell+1-n)}^\infty\sum_{k= n+j}^\infty ,$$ 
we get from Proposition \eqref{propProjPnC}
	\begin{align} \label{PnHAell}
	P_n\Cc f(z)
	&	=  \sum_{j=0\vee (1-n)}^\infty  \sum_{  k=0}^{n+j-1} \frac{(-1)^{n +  k+1}  \Gamma(j+ n   ) }{ 2^{n+j}n! \Gamma(n+j-k)}  H_{n+j-k-1,n} (z,\bz)  \alpha_{j,k}.
	\end{align}
	So that for $f\in \mathcal{A}_{\ell}^{2}$, we have $ \alpha_{j,k}=0$ for every $k\ne \ell$ and the previous expression \eqref{PnHAell} reduces further to
	\begin{align*}
	P_n\Cc f(z) 	
	&	=  \sum_{j= 0\vee (\ell+1-n) }^\infty   \frac{(-1)^{n +  \ell+1}  \Gamma(j+ n   ) }{ 2^{n+j}n! \Gamma(n+j-\ell)}  H_{n+j-\ell-1,n} (z,\bz)  \alpha_{j,\ell}.
	\end{align*}
	This proves $(i)$.

	Similarly, for the case of $f\in \widetilde{\mathcal{A}}_{\ell}^{2}$, we have $ \alpha_{j,k}=0$ for every $j\ne \ell$, so that \eqref{PnHAell} gives rise to
	$P_0\Cc f(z) = 0$ when $\ell=n=0$, and therefore
	\begin{align} \label{PnHAell}
	P_n\Cc f(z)
	&	=    \sum_{  k=0}^{n+\ell-1} \frac{(-1)^{n +  k+1}  \Gamma(\ell+ n   ) }{ 2^{n+\ell}n! \Gamma(n+\ell-k)}  H_{n+\ell-k-1,n} (z,\bz)  \alpha_{\ell,k}
	\end{align}
	whenever $n+\ell >0$.
	For any $\ell$ such that $n+\ell >0$, the $k$ such that $k \leq n +\ell -1$ satisfy $ n+\ell-1-k \geq 0$. 
\end{proof}

\begin{remark}
	The spaces
	$R^\ell_n$ (resp. $\widetilde{R}^\ell_n$); $n,\ell=0,1,2,\cdots,$ are pairwise orthogonal with respect to $n$, independently of $\ell$, and form a decreasing (resp. increasing ) sequence of in $\ell$, for fixed $n$. Indeed, we have $R^\ell_{n} \supset  R^{\ell+1}_{n}$ and 
	$ \widetilde{R}^\ell_{n} \subset  \widetilde{R}^{\ell+1}_{n}.$
\end{remark}

The characterization of the full range of $\Cc$ on $\LgC$ by means of their $k$-Bergman projections $P_n$ requires further investigations. However, by considering the spaces
$$E_{j}^+   =  \overline{span\{\psi_{n, n+j}; \, n=0,1,2, \cdots \}} $$
and
$$E_{j}^- = \overline{span\{ \psi_{n+j,n }; \, n=0,1,2, \cdots \}},$$
for given nonnegative integer $j$, 
we can prove the spaces $ E_\ell$, for varying integer $\ell
	 = \cdots, -2,-1,0,1,2,\cdots ,$ defined by  
$E_\ell = E_{|\ell|}^+$ for $\ell \geq 0$ and $E_\ell = E_{|\ell|}^-$ when $\ell< 0$, 
form an orthogonal Hilbertian decomposition of the range of the weighted Cauchy transform $\Cc$ in  $\LgC$.

\begin{theorem}	We have 
	$$\Cc(\LgC) = \bigoplus_{\ell\in\Z} E_\ell .$$
\end{theorem}

\begin{proof}
We begin by showing that the spaces $E_\ell $ are mutually orthogonal in $\LgC$. By direct computation, using \eqref{actionCH}, \eqref{16negative} and Fubini's theorem, infers 
\begin{align}
		\scal{\psi_{m,n},\psi_{j,k}}
		&= 	\int_{\C} \overline{\psi_{m,n}}\psi_{k,j}(z) e^{-|z|^2} d\lambda(z) \\
	& = \frac{ c_{m-1,n} c_{j-1,k} }{2}	 A_{m,n,j,k} \int_0^{\infty}
	\frac{t^{m+n-1}}{t^{\min(m-1,n)+\min(j-1,k)}} R_{m,n,j,k}(t) e^{-3t}dt , \label{scalarChCh}
	\end{align}
	where  $c_{m,n}$  and $ R_{m,n,j,k}$  are as in \eqref{constcmn} and \eqref{ptodHyp}, respectively. 
This shows that
	the orthogonality of the system $(\psi_{n,m})_{m,n}$ in $\LgC $ is equivalent to the nullity of the angular part given by
	$ A_{m,n,j,k} 
= 2\pi \delta_{m-j,n-k} .$
		Thus, for $m-j=\ell \ne n-k =\ell'$, we have
	$	\scal{\psi_{n,m},\psi_{k,j}}_{\Hq} =0$. This proves in particular that the $E_\ell$; $\ell\in \Z$, form an orthogonal sequence in $\LgC$. 
	Subsequently, from general theory of functional analysis, the orthogonal sum $\bigoplus_{\ell\in\Z} E_\ell$ is a closed subspace of $\LgC$. 
		The inverse inclusion is clear.
\end{proof}

\end{document}